\documentclass[oneside,draft,A4,11pt]{amsart}
\usepackage{enumitem}
\usepackage{graphicx}
\usepackage[margin=1.4in]{geometry} 
\renewcommand{\sc}{\scshape}
\usepackage{amssymb}
\usepackage{amsthm}

\usepackage{amsmath}
\usepackage{color}
\usepackage[utf8]{inputenc}
\newcommand*\Laplace{\mathop{}\!\mathbin\bigtriangleup}

\usepackage{tikz}
\usepackage[all]{xy}
\usepackage[bookmarksnumbered,colorlinks]{hyperref}
\usepackage{dsfont}
\usepackage[normalem]{ulem}
\usepackage{float} 
\def\br#1\er{\textcolor{red}{#1}} 
\newcommand{\R}{\mathds{R}}

\usepackage{autonum}
\begin{document}
\title{A Moser-Bernstein problem for Riemannian warped products}
\author[A.L. Albujer]{Alma L. Albujer} \address{Departamento de
  Matemáticas, Edificio Albert Einstein\hfill\break\indent Universidad
  de Córdoba, Campus de Rabanales,\hfill\break\indent 14071 Córdoba,
  Spain}

\author[J. Herrera]{J\'onatan Herrera} \address{Departamento de
  Matemáticas, Edificio Albert Einstein\hfill\break\indent Universidad
  de Córdoba, Campus de Rabanales,\hfill\break\indent 14071 Córdoba,
  Spain}

\author[R. Rubio]{Rafael M. Rubio} \address{Departamento de
  Matemáticas, Edificio Albert Einstein\hfill\break\indent Universidad
  de Córdoba, Campus de Rabanales,\hfill\break\indent 14071 Córdoba,
  Spain}

\newtheorem{thm}{Theorem}[section]
\newtheorem{theorem}{Theorem}[section]
\newtheorem{proposition}[thm]{Proposition} \newtheorem{lemma}[thm]{Lemma}
\newtheorem{corollary}[thm]{Corollary} \newtheorem{conv}[thm]{Convention}
\theoremstyle{definition} \newtheorem{defi}[thm]{Definition}
\newtheorem{notation}[thm]{Notation} \newtheorem{exe}[thm]{Example}
\newtheorem{conjecture}[thm]{Conjecture} \newtheorem{prob}[thm]{Problem}
\newtheorem{remark}[thm]{Remark}
\newtheorem{example}[thm]{Example}
\newcommand{\nablat}{\overline{\nabla}}
\renewcommand{\div}{\mathrm{div}}
\begin{abstract}In this work we deal with an elliptic non-linear problem, which arises naturally from Riemannian geometry. This problem has clasically been studied in the the Euclidean $n$-dimensional space and it is known as the Moser-Bernstein problem. Nevertheless we solve this type of problems in a wide family of Riemannian manifolds, constructed as Riemannian warped products. More precicely, we study the entire solutions to the minimal hypersurface equation in a Riemannian warped product $M=P\times_h\mathds{R}$, where $P$ is a complete Riemannian parabolic manifold and $h$ a positive smooth function on $P$.
\end{abstract}



\maketitle
\usetikzlibrary{matrix}

\section{Introduction}
\label{sec:introduction}
Let $(P^n,\sigma)$ be a Riemannian $n$-dimensional manifold and $h\in\mathcal{C}^\infty(P^n)$ a given function. The aim of this paper is to study  uniqueness results for entire solutions to the partial differential equation
\begin{equation}\label{eq:PDE}
\div \left( \dfrac{h \nabla u}{\sqrt{1+h^2 \left| \nabla u \right|_{\sigma}^2 }} \right)+ \dfrac{\sigma(\nabla h, \nabla u)}{\sqrt{1+h^2 \left| \nabla u \right|_{\sigma}^2 }}=0,
\end{equation}
whose $\left| \nabla u\right|_{\sigma}$ is bounded and
where $\div$ and $\nabla$ stand for the divergence and gradient operator of $(P^n,\sigma)$, and $\left|\cdot\right|_{\sigma}$ denotes its related norm.

The above equation is a non-linear elliptic equation of divergence form (see \cite{Gilbarg_1983}) which has a remarkable geometrical meaning.
In the particular case where $P^n=\mathds{R}^n$ is the Euclidean space and $h\equiv 1$, equation~\eqref{eq:PDE} becomes the minimal hypersurface equation in the Euclidean space, 

\begin{equation}\label{eq:PDE_minimal}
  \div \left( \dfrac{\nabla u}{\sqrt{1+\left| \nabla u \right|_{\sigma}^2 }} \right)=0.
\end{equation}

The graph of any solution to~\eqref{eq:PDE_minimal} is a hypersurface in $\mathds{R}^n$ with zero mean curvature.  Those hypersurfaces are called minimal, and it is a classical fact that they are critical points of the volume functional $\int\sqrt{1+|\nabla u|_{\sigma}^2}\,dV$ under normal variations, where $dV$ denotes the canonical volume form on the Euclidean space.

Since in 1914  Bernstein~\cite{Bernstein_1914} proved that the only entire solutions to the minimal surface equation~\eqref{eq:PDE_minimal} in $\mathds{R}^3$ are the affine functions, the study of existence and uniqueness of minimal hypersurfaces in different ambient spaces and/or general dimension has become a topic of wide interest in the theory of Riemannian submanifolds. 

The possible extension of the Bernstein theorem to higher dimension is known in the literature as the Bernstein conjecture. Giving a partial answer to this conjecture, Moser~\cite{Moser_1961} proved in 1961 that the only entire solutions $u$ to the minimal hypersurface equation in $\mathds{R}^{n+1}$ such that $|\nabla u|_{\sigma}\leq C$ for some constant $C > 0$, are the affine functions. This result is usually known as the Moser-Bernstein theorem.

Some years later Simons~\cite{Simons_1968}, jointly with some previous results by Almgren~\cite{Almgren_1966} and De Giorgi~\cite{deGiorgi_1965}, extended the classical Bernstein theorem up to dimension $n\leq 7$. Futhermore, Bombieri, de Giorgi and Giusti~\cite{Bombieri_1968} presented in 1968 counterexamples to this result for each $n\geq 8$. Consequently, in order to obtain some uniqueness results for minimal hypersurfaces in general dimension some extra assumptions are needed, as it is the case of the Moser-Bernstein theorem.

As a first generalization of the ambient space, given an $n$-dimensional Riemannian manifold $(P^n,\sigma)$ we can consider the product manifold $P^n\times\mathds{R}$ endowed with the metric
\begin{equation}\label{eq:metric_prod}
g=\pi_{P}^\ast(\sigma)+\pi_{\mathds{R}}^\ast (dr^2),
\end{equation}
where $\pi_P$ and $\pi_\mathds{R}$ stand for the natural projections of $P\times \mathds{R}$ onto $P$ and $\mathds{R}$ respectively. In this situation, a function $u\in\mathcal{C}^\infty(P)$ determines a minimal hypersurface in $P^n\times\mathds{R}$ if and only if $u$ satisfies the same partial differential equation as in~\eqref{eq:PDE_minimal}. 

In this context of Riemannian product spaces, and in the particular case where $n=2$ and the fiber $P$ is a complete Riemannian surface with non negative Gaussian curvature, Rosenberg~\cite{Rosenberg_2002} showed in 2002 that any entire minimal graph in $P^2\times\mathds{R}$ must be totally geodesic. A few years later, in 2007 Al\'ias, Dajzcer and Ripoll~\cite{Al_as_2006} improved this result by showing that in the case where $P$ is a complete Riemannian surface with non-identically zero non negative Gaussian curvature the only entire minimal graphs in $P^2\times\mathds{R}$ are the slices $P^2\times\{r_0\}$, $r_0\in\mathds{R}$. This assumption on the Gaussian curvature of the fiber is necessary, as it is shown by the fact that there exist plenty of examples of entire minimal graphs in $\mathds{H}^2\times\mathds{R}$, $\mathds{H}^2$ being the hyperbolic plane, see for instance Example \ref{example}.

In general dimension, Rosenberg, Schulze and Spruck~\cite{Rosenberg_2013} proved that given $P^n$  a complete Riemannian manifold with non negative Ricci curvature and sectional curvature bounded from below, any entire minimal graph in $P^n\times\mathds{R}$ with non negative height function must be a slice. On the other hand, Oliveira and de Lima~\cite{Oliveira_2015} have obtained a new uniqueness result for constant mean curvature hypersurfaces in this context. Specifically, under the same assumptions on $P^n$ as in~\cite{Rosenberg_2013}, they have proven that any constant mean curvature entire graph in $P^n\times\mathds{R}$ with bounded second fundamental form and bounded gradient is necessarily minimal. If, in addition, the graph is bounded from below it must be a slice. However, due to their assumption on the second fundamental form, the result cannot be formulated in terms of any elliptic differential equation.

The assumption of boundedness of the height function of the graph is closely related to the well-known half-space property. A manifold $P$ is said to have the half-space property if any minimal hypersurface properly immersed in $P\times\mathds{R}_{+}$ is a slice. In this sense Hoffman and Meeks~\cite{Hoffman_Meeks_1990} proved in 1990 the well-known half-space theorem, which states that any proper immersed minimal surface in $\mathds{R}^3=\mathds{R}^2\times\mathds{R}_{+}$ is necessarily a slice.

Recently, Romero, Rubio and Salamanca~\cite{Romero_Rubio_Salamanca} have considered another generalization of the ambient spaces, considering a Riemannian warped product space $P_{\,\,h}\!\times I$ with base an open interval $(\mathds{R},dr^2)$, fiber a Riemannian manifold $(P,\sigma)$ and warping function a smooth function $h\in\mathcal{C}^\infty(I)$. That is, $P_{\,\,h}\!\times I$ denotes the product manifold $P\times I$ endowed with the Riemannian metric
\begin{equation}\label{eq:metric_grw}
g=h(\pi_I)^2\pi_P^\ast(\sigma)+\pi_I^\ast(dr^2).
\end{equation}

The graph determined by a smooth function $u\in\mathcal{C}^\infty(P)$ such that $u(P)\subseteq I$ is a minimal hypersurface in  $P_{\,\,h}\!\!\times I$ if and only if 
\begin{equation}\label{eq:PDE_grw}
\div \left( \dfrac{\nabla u}{h(u)\sqrt{h(u)^2+\left| \nabla u \right|_{\sigma}^2 }} \right)=\frac{h'(u)}{\sqrt{h(u)^2+\left|\nabla u\right|_{\sigma}^2}}
\left\{n-\frac{\left|\nabla u\right|_{\sigma}^2}{h(u)^2}\right\}.
\end{equation}
In order to obtain Moser-Bernstein type results for entire graphs in $P_{\,\,h}\!\times I$ the authors ask in~\cite{Romero_Rubio_Salamanca} the fiber to be parabolic, and need some extra assumptions on the warping function. Specifically they obtain some interesting results in the case where $h$ is not constant on any non-empty open subset of $I$, $\ln h$ is convex and the fiber $P$ is parabolic.

Along this paper we will deal with a different family of Riemannian warped product spaces. Given $(P,\sigma)$ a Riemannian manifold and $h\in\mathcal{C}^\infty(P)$, let us consider the product manifold $(M,g)$ given by 
\begin{equation}\label{eq:metric}
M=P\times \R,\qquad g=\pi^\ast_P(\sigma)+h(\pi_P)^2\pi_I^\ast(dr^2).
\end{equation}
This warped product is usually denoted by $P\times_h \mathds{R}$. From~\cite[Subsection 2.2]{Dajczer_Hinojosa_deLira_2008} it is immediate to observe that a smooth function $u\in\mathcal{C}^\infty(P)$ determines an entire minimal graph in $P\times_h \mathds{R}$ if and only if $u$ satisfies equation~\eqref{eq:PDE}. Thus we will refer to~\eqref{eq:PDE} as the minimal hypersurface equation in the warped product  $P\times_h \mathds{R}$.

Several authors have studied certain uniqueness results in this new context. Specifically, Romero and Rubio \cite{Romero_Rubio_2016} have recently obtained uniqueness results for the solutions to the minimal hypersurface equation in the case where $P$ is a compact manifold. Furthermore, Dajzcer and de Lira~\cite{DajczerdeLira} showed that under certain curvature conditions of the ambient space and a second order regularity on the warping function, bounded graphs with constant mean curvature in $P\times_h \mathds{R}$ must be totally geodesic slices.

Our main goal is to obtain a new Moser-Bernstein result for the solutions to~\eqref{eq:PDE} in the case where $(P,\sigma)$ is a parabolic Riemannian manifold, Theorem~\ref{thm:MB}. It is interesting to observe that in order to get this result, we only need to ask for an appropriate boundedness condition on the warping function. In fact Theorem~\ref{thm:MB} is obtained as a consequence of a more general geometric result, Theorem~\ref{thm:2}, which provides sufficient conditions to guarantee the uniqueness of entire half-bounded graphs in $P\times_h\mathds{R}$ with bounded gradient and signed (not necessarily constant) mean curvature function. In order to prove Theorem~\ref{thm:2} we make use of two main ideas. On the one hand, the fact that the height function of a graph in $P\times_h\mathds{R}$ with non positive mean curvature function is superharmonic for a convenient conformal change of the metric. On the other hand, we have into account that parabolicity, which is not in general invariant under conformal transformations, is preserved under the so-called quasi-isometries. In particular, we will be able to ensure the parabolicity of the graph under some mild conditions on both, the warping function $h$ and the gradient of the graph.

\section{Preliminaries}
\label{sec:preliminaries}

Let $(M, g)$ be an $(n+1)$-dimensional Riemannian manifold endowed with a nowhere zero Killing vector field $K$. Suppose  that the orthogonal distribution $K^{\perp}$ is integrable or, equivalently, that $\omega\wedge d\omega=0$ where $\omega$ is the $1$-form  metrically equivalent to $K$. Given an integral leaf $P$ of $K^{\perp}$, consider a bounded domain $\Omega\subset P$ with regular boundary $\Gamma =\partial\Omega$. Let $\phi: I\times \overline\Omega \to M$ be the flow generated by $K$ with initial values on $\overline\Omega$, where $I$ is the maximal interval of definition. Since $K$ is a Killing vector field $\phi$ is a local isometry, thus the Riemannian manifold $M$ is locally isometric to $\Omega\times_{h} I$ where $h = \| K\|$ and $\|\cdot\|$ denotes the norm on $(M,g)$.

If in addition the Riemannian manifold is $1$-connected and the Killing vector field is complete, it is not difficult to see that the vector field $K$ is parallel relative to the metric $\bar{g}=\frac{1}{g(K,K)}g$. As a consequence the $1$-form metrically equivalent to $K$ with this new metric, $\omega=\bar{g}(K,\cdot)$, must be closed. Therefore, Poincar\'e's lemma ensures the existence of a smooth function $l:M\rightarrow \mathds{R}$, such that $dl=\omega$.

We can take then the integral leaf $P=l^{-1}(0)$. If $\varphi_{p}(r)$ is the global flow of the vector field $K$ through a point $p$ (and so, with $r\in \mathds{R}$), then $\frac{d}{dr}\big(l(\varphi_{p} (r)\big)=1$. This guarantees that the integral curves of $K$ cross one (and only one) time $P$.

Therefore, we can extend the map $\phi$ to a global isometry, obtaining
\[\Phi : P\times_{h} \mathds{R}\longrightarrow M, \quad  \Phi(p,r)=\varphi_{p}(r).\]
As a direct consequence the Riemannian manifold $(M,g)$ is a warped product $P\times_{h} \mathds{R}$, with $h=\| K\|$.

Taking into account the previous remarks, from here on we will identify $(M,g)$ with a warped product $P\times_{h} \mathds{R}$ as defined on \eqref{eq:metric}. From construction, the vector field $\partial_r:=\frac{\partial}{\partial r}$ is the associated Killing vector field to $P\times_{h} \mathds{R}$, so
\begin{eqnarray}
\label{eq:2}
g(\nablat_X\partial_r,X)=0,\qquad \hbox{for any vector field $X\in\mathfrak{X}(M)$.}
\end{eqnarray}
where $\nablat$ stands for the Levi-Civita connection in $M$.

Let $x:\Sigma\longrightarrow M$ be an isometric immersion of an $n$-dimensional connected manifold $\Sigma$ in $M$. If the normal bundle of the immersion is trivial, i.e., if there exists a globally defined unitary normal vector field $N$ on $\Sigma$, the hypersurface is said to be two-sided. When the ambient $M$ is orientable, this property is equivalent to the orientability of $\Sigma$. In particular, if we consider a Riemannian warped product $M=P\times_h\R$ and a hypersurface $\Sigma$, which is transverse to the the Killing vector field $\partial_{r}$, then the hypersurface is locally a graph on a suitable domain included in $P$, and therefore we can take $N$ such that $g(N,\partial_r)$ is signed.

From now on, we will focus in the case where $\Sigma$ is an entire graph on $P$. So, we will consider an embebded hypersurface $\Sigma_u$ in $M$, given by $\Sigma_u=\{(p,u(p)): p\in P\}$, where $u\in C^\infty(P)$. It is immediate to see that a unitary normal vector to $\Sigma_u$ may be chosen as 

\begin{equation}
\label{eq:normal}
N=\frac{h}{\sqrt{1+h^2\left| \nabla u\right|_{\sigma}^2}}\left(\frac{1}{h^2}\partial_r-\nabla u\right).
\end{equation}

Let  $A: \mathcal{X}(\Sigma)\rightarrow \mathcal{X}(\Sigma)$ be the Weingarten operator of the hypersurface defined by $A(X)=-\overline{\nabla}_X N$. Recall that given $\{E_1,\dots,E_n\}$ a local orthonormal frame on $\Sigma_u$, the mean curvature function of the hypersurface is given by:
\begin{equation}
  \label{eq:8}
  H=\dfrac{1}{n} \sum_{i=1}^n g(A(E_i),E_i).
\end{equation}
In particular, the mean curvature of the graph $\Sigma_u$ satisfies
\[nH(u)=\mathrm{div} \left(  \frac{h \nabla u}{ \sqrt{1 + h^2 \left|\nabla u\right|_{\sigma}^2}} \right) +\frac{\sigma (\nabla h, \nabla u)}{\sqrt{1 + h^2\left|\nabla u\right|_{\sigma}^2}} \cdot\]

Finally let us notice that the Riemannian manifold $(M,g)$ is foliated by the family of embedded hypersurfaces $P\times\{r_0\}$, with $r_0\in \mathds{R}$ constant, which are totally geodesic and consequently have constant mean curvature $H\equiv 0$.

\section{Setting up}

\subsection{Parabolic Riemannian manifolds}

Let us recall that a function $u\in C^{\infty}(P)$ is called \emph{superharmonic} if $\Laplace u\leq 0$. Then, by definition, a complete (non-compact) Riemannian manifold $(P,\sigma)$ is parabolic if it admits no positive superharmonic functions but the constants (see \cite{Kazdan_1987}). In the case of dimension two, parabolicity is closely related with the behaviour of the Gaussian curvature. For example, a well-known result of Ahlfors and Blanc-Fiala-Huber states that any complete Riemannian surface with non negative Gauss curvature is parabolic (\cite{Huber_1958}). 
 In a 
weaker version, if the Gauss curvature of a complete Riemannian surface satisfies $K  \geq \frac{-1}{r^2 \, \ln r}$, being $r$ the distance to a sufficiently large fixed point, then the surface must be parabolic \cite{Greene_1979}. In the same direction, 
in \cite{Li_2002} it is shown that any complete Riemannian surface with finite total curvature  must be parabolic. Nevertheless, in the $n>2$ dimensional case parabolicity has no clear relationship with the sectional curvature. Indeed, the Euclidean 
space $\mathds{R}^n$ is parabolic if and only if $n \leq 2$. Instead, parabolicity has a remarkable relation with the volume growth of the geodesic balls on any complete (non-compact) Riemannian manifold $(M,g)$. In fact, if $(M,g)$ has moderate volume growth, then it must be parabolic (\cite{Karp_1982}).

Parabolicity is not, in general, a conformal invariant. However, it is invariant under the so-called \emph{quasi-isometries} (see  \cite{Grigor_yan_1999}). Concretely, two Riemannian manifolds $(P,\sigma)$ and $(P',\sigma')$ are said to be quasi-isometric if there exists a diffeomorphism $\varphi:P\rightarrow P'$ and a constant $a$ such that, for all $v\in TP$ it holds 
\begin{equation}
  \label{eq:9}
  a^{-1}\left| v \right|_{\sigma}\leq \left| d\varphi(v) \right|_{\sigma'}\leq a \left| v \right|_{\sigma}.
\end{equation}

Then, the following result follows,

\begin{proposition}(\cite[Corollary 5.3]{Grigor_yan_1999})
\label{thm:3}
If two manifolds $(P,\sigma)$ and $(P',\sigma')$ are quasi-isometric, then both are parabolic or not simultaneously.
\end{proposition}

Now, we can state the following technical result for the case of entire  graphs in $P\times_h\mathds{R}$, where $(P,\sigma)$ is a parabolic manifold.

\begin{proposition}
\label{thm:1}
Let $(P,\sigma)$ be a parabolic $n$-dimensional Riemannian manifold, any smooth positive function and let us consider the $(n+1)$-Riemannian manifold $P\times_h\mathds{R}$. If it exists a smooth function $u\in\mathcal{C}^\infty (P)$ satisfying
\begin{equation}
  \label{eq:10}
  \left| \nabla u \right|_{\sigma}\leq \dfrac{B}{h},
\end{equation} for some positive constant $B\in \mathds{R}$, then the hypersurface $\Sigma_u$ obtained as the entire graph determined by $u$ is also parabolic with respect to the induced metric from $P\times_h \mathds{R}$. 
\end{proposition}

\begin{proof}
  Let us consider a function $u\in C^{\infty}(P)$ satisfying the hypothesis. Our aim is to show that both, $\Sigma_u$ with the induced metric from $P\times_h \mathds{R}$, and $(P,\sigma)$ are quasi-isometric.

  It is a quite straightforward computation to show that $(\Sigma_{u},g|_{\Sigma_{u}})$ is isometric to $(P,\sigma')$ with $\sigma'=h^2du \otimes du + \sigma$. Then, for any arbitrary point $p\in P$ and any vector $v\in T_{p}P$, we have

  \begin{equation}
    \label{eq:11}
    \sigma'(v,v)=h^2\sigma(\nabla u,v)^2+\sigma(v,v)\geq \sigma(v,v).
  \end{equation}

  On the other hand, making use of the inequality $\left| \nabla u \right|_{\sigma}\leq B/h$, we  obtain

  \begin{equation}
    \label{eq:12}
    \sigma'(v,v)=h^2\sigma(\nabla u,v)^2+\sigma(v,v)\leq h^2 \left| \nabla u \right|_{\sigma}^2\sigma(v,v)+\sigma(v,v)\leq (B^2+1)\sigma(v,v)
  \end{equation}
Joining \eqref{eq:11} and \eqref{eq:12} together  we deduce that $(\Sigma_u,g|_{\Sigma_u})$ and $(P,\sigma)$ are quasi-isometric, and the result follows.

\end{proof}

\subsection{The mean curvature of the graph}
Let us now consider $\Sigma_{u}$ an entire graph as before and the function $\tau=\pi_\mathds{R}|_{\Sigma_{u}}:\Sigma_{u}\rightarrow \mathds{R}$. A simple calculation shows that

\begin{eqnarray}
\label{eq:1}
\nabla^{\Sigma} \tau = \dfrac{1}{h^2}\partial_r^{T},
\end{eqnarray}
where $\partial_r^T$ denotes the orthogonal projection of $\partial_r$ to $T\Sigma$, and $\nabla^{\Sigma}\tau$ denotes the gradient of $\tau$ with respect to the induced metric from $\Sigma_u$.

There is a nice relation between the mean curvature of $\Sigma$ and the laplacian of $\tau$, $\Delta^\Sigma\tau$. In fact, given $\left\{E_1,\dots,E_n  \right\}$ an orthonormal frame on $\Sigma$, it follows that:

\begin{equation}
\label{eq:4}
\begin{array}{rl}
\Laplace^{\Sigma} \tau =  & \mathrm{div^{\Sigma}}(\nabla^{\Sigma} \tau)\\
  = & \sum_{i=1}^{n}g(\nabla^{\Sigma}_{E_i}\nabla^{\Sigma}\tau,E_i)  \\
 = & \sum_{i=1}^ng(\nabla^{\Sigma}_{E_i}\left(\dfrac{1}{h^2}\partial_r^T  \right),E_i) \\
                           = &  g(\partial_r^T,\nabla^{\Sigma} \left(1/h^2 \right)) + \dfrac{1}{h^2} \sum_{i=1}^n g \left(\nabla^{\Sigma}_{E_{i}}\partial_r^T,E_i  \right), \\
\end{array}
\end{equation}
$\mathrm{div}^\Sigma$ and $\nabla^\Sigma$ being the divergence operator and the Levi-Civita connection in $\Sigma_u$.

Now observe that, by recalling \eqref{eq:1}, the first term can be expressed as:

\begin{equation}
  \label{eq:3}
  g(\partial_r^T,\nabla^{\Sigma} \left( 1/h^2 \right))=h^2 g(\nabla^{\Sigma} \tau,\nabla^{\Sigma}(1/h^2))=-2g(\nabla^{\Sigma} \tau, \nabla^{\Sigma}(\ln h)).
\end{equation}

For the second term, and considering $N$ a unitary normal vector field for $\Sigma$,

\begin{equation}
  \label{eq:6}
  \begin{array}{rl}
  \sum_{i=1}^ng(\nabla^{\Sigma}_{E_i}\partial_r^{T},E_i)=&\sum_{i=1}^n g \left( \overline{\nabla}_{E_i} \left( \partial_r-g(N,\partial_{r})N \right),E_i \right)\\ =& g(N,\partial_{r})\sum_{i=1}^n g(A(E_i),E_i)\\ = & n\,H\,g(N,\partial_{r})
  \end{array}
\end{equation}

In conclusion,
\begin{equation}
  \label{eq:7}
  \Laplace^{\Sigma} \tau +2 g(\nabla^{\Sigma} \tau,\nabla^{\Sigma} (\ln h)) = n\,H\,g(N,\partial_r).
\end{equation}

\section{Main Results}
\label{sec:main-results}

Making use of the previous equation \eqref{eq:7} and of an appropriate conformal change of metric, which is the key of our next result, we can state the following theorem.

\begin{theorem}
\label{thm:2} Let $(P,\sigma)$ be a parabolic Riemannian manifold, and consider $h:P\rightarrow \mathds{R}$ a positive smooth function satisfying
\begin{equation}
  \label{eq:14}
  0<\inf(h)\leq \sup(h)<\infty.
\end{equation}

Let $\Sigma_u$ be an entire graph on $P\times_h\R$ determined by a bounded from below (resp. bounded from above) function $u\in\mathcal{C}^{\infty}(P)$, such that $|\nabla u|_\sigma <C$ for a certain $C\in\R$ and satisfying $H\leq 0$ (resp. $H\geq 0$). Then $\Sigma_u$  must be a slice, i.e., $\Sigma_u=P\times \left\{ r_0 \right\}$ for some constant $r_{0}\in \mathds{R}$. 	
	
In particular, slices are the only bounded entire graphs with bounded gradient and signed mean curvature.
 \end{theorem}

\begin{proof} 
Let us assume without loss of generality that $\Sigma_u$ is bounded from below  with $H\leq 0$, since the other case is analogous.

Let us begin by assuming that $n\geq 3$ and consider $u:P\rightarrow \mathds{R}$ a function satisfying  the assumptions of the theorem. Let the normal vector field to $\Sigma_u$ being defined as in~\eqref{eq:normal}. Then, from \eqref{eq:7} it follows that:
\[\Laplace^{\Sigma} \tau + 2g(\nabla^{\Sigma} \tau,\nabla^{\Sigma}(\ln h))\leq 0.\]   
Now, let us recall that, under a conformal change $\tilde{g}=\varphi g$ the Laplace operator  transforms as

\begin{equation}
  \label{eq:13}
  \tilde{\Laplace}^{\Sigma}f= \dfrac{1}{\varphi}\left( \Laplace^{\Sigma} f +\dfrac{n-2}{2}g(\nabla^{\Sigma} f,\nabla^{\Sigma} (\ln \varphi)) \right),
\end{equation}
(see for instance \cite{Besse_1987}). Then, by considering the conformal factor $\varphi=h^{\frac{4}{n-2}}$, it follows that
\[\tilde{\Laplace}^{\Sigma}\tau=\dfrac{1}{h^{\frac{4}{n-2}}}\left( \Laplace^{\Sigma} \tau +2 g(\nabla^{\Sigma} \tau,\nabla^{\Sigma}(\ln(h)))\right)\leq 0.
\]
Hence, the function $\tau:\Sigma\rightarrow \mathds{R}$ is $\tilde{g}$-superharmonic.

Finally, as $P$ is parabolic, $\left| \nabla u \right|_{\sigma}$ is bounded and $\sup(h)<\infty$, Proposition \ref{thm:1} ensures that $\Sigma$ is also parabolic with respect to $g$. Moreover, from \eqref{eq:14}  it easily follows that $g$ and $\tilde{g}$ are quasi-isometric, and then, that $\Sigma_u$ is also parabolic with respect to $\tilde{g}$. As there exists a certain real constant $d$, such that $\tau+d$ is a positive  superharmonic function on a parabolic manifold,  then the function $u\equiv\tau$ must be necessarily constant. In conclusion, $\Sigma_u=P\times\left\{r_0 \right\}$ for some constant $r_0\in\mathds{R}$.

It only remains the case of dimension $n=2$, since the previous conformal change is no longer well defined. However, we can consider $\overline{P}=P\times \mathds{S}^1$ endowed with the metric $\overline{\sigma}=\sigma + d\theta^{2}$ (where $d\theta^2$ denotes the usual metric on $\mathds{S}^1$). The following observations are in order:
\begin{enumerate}[label=(\roman*)]
\item Firstly, due the fact that $\mathds{S}^1$ is compact, the parabolicity of $P$ implies that the product manifold $P\times \mathds{S}^1$ is also parabolic, (see \cite{Kazdan_1987}).
  
\item Secondly, if $\Sigma$ is a  hypersurface with non positive mean curvature in $P\times_{h} \mathds{R}$, then $\overline{\Sigma}=\Sigma\times \mathds{S}^1$ is also a hypersurface in $\overline{P}\times_h\mathds{R}$ with the same property.
  
\item Finally, if $\Sigma$ is the graph determined by a function $u:P\rightarrow\mathds{R}$, then $\overline{\Sigma}= \Sigma\times \mathds{S}^1$ is the graph in $\overline{P}\times_h\mathds{R}$ defined by the function $\overline{u}\in C^{\infty}(\overline{P})$ given by $\overline{u}(x,y)=u(x)$ for $(x,y)\in P\times \mathds{S}^1$.
\end{enumerate}

Hence, given  an entire graph $\Sigma_u$ in $P^2\times_h\mathds{R}$ with bounded $\left| \nabla u \right|_{\sigma}$ and non positive mean curvature, we can consider $\overline{\Sigma}_{\overline{u}}$ as above,  which also has non positive mean curvature in $\overline{P}\times \mathds{R}$. Due the fact that now $\overline{P}$ is $3-$dimensional, we are in conditions to apply the previous arguments, ensuring that $\overline{u}$, and so $u$, is constant. 
\end{proof}

As an immediate consequence of the previous result, we have the following corollary regarding minimal graphs:

\begin{corollary}\label{cor:minimal}
  If $\Sigma$ is a minimal entire graph defined by a half-bounded function $u\in\mathcal{C}^\infty(P)$ (i.e., bounded from below or from above)  with bounded gradient, then $\Sigma=P\times \left\{ r_0 \right\}$ for some constant $r_0\in \mathds{R}$.
\end{corollary}

Some final remarks regarding the parabolicity condition are in order. On the one hand, let us observe that the parabolicity of the base $P$ is necessary as there exist counterexamples even in the case where $h\equiv 1$. 

\begin{example}\label{example}
In the case where $P=\mathds{H}^2$ and $h$ is a constant function, Nelli and Rosenberg proved in~\cite{Nelli_Rosenberg_2002} the existence of non-trivial bounded minimal graphs in $\mathds{H}^2\times\mathds{R}$ with bounded gradient.

Specifically, considering the disk model $D=\{0\leq x_1^2+x_2^2<1\}$ for $\mathds{H}^2$, the boundary of the product space $\mathds{H}^2\times\mathds{R}$ is the cylinder $\mathds{S}^1\times\mathds{R}$. Then, in~\cite[Theorem 4]{Nelli_Rosenberg_2002} it is shown that given any continuous rectificable Jordan curve $\Gamma$ in $\mathds{S}^1\times\mathds{R}$, there exists a unique minimal graph on $\mathds{H}^2$ having $\Gamma$ as asymptotic boundary. This graph is obtained as a limit of appropriate graphs defined on disks $D_n$ centered at the origin of Euclidean radius $1-\frac{1}{n}$, and it is a standard fact that the hyperbolic gradient of the limit function tends to zero at $\partial D$, so in particular it is bounded.
\end{example}

On the other hand, the parabolicity of $P$ induces some restrictions on the geometry of the Riemannian manifold $P\times_h \mathds{R}$. In fact, let us consider in $P\times_h \mathds{R}$  a hypersurface $\Sigma_u$ given as the graph of a smooth function $u\in P$. Consider on $\Sigma_u$ the distinguished function $\Theta:=g(\partial_r, N)$, where $N$ is defined by~\eqref{eq:normal}. Assuming that $\Sigma_u$ is minimal, and taking into account \cite[Proposition 6]{Al_as_2006}, we arrive to

\begin{equation}
\Laplace \Theta=-({\rm trace}(A^2)+\overline{{\rm Ric}}(N,N)),\label{eq:5}
\end{equation}
where $\overline{{\rm Ric}}$ denotes the Ricci tensor on $M$. Consequently, if the entire graph $\Sigma_u$ is parabolic and the Ricci curvature of the spacetime is non negative in the direction of the normal vector field $N$, we can conclude that $\Theta$ is constant.

Observe however that we can take now $\Sigma_u=P$, which is by hypothesis a parabolic manifold. Now, if we consider a non-constant warping function $h$, we find a contradiction. In fact, every slice $P\times \{r_{0}\}$ for $r_{0}\in \mathds{R}$ is a totally geodesic hypersurface with normal  $N=\dfrac{1}{h}\partial_{r}$. Hence $\Theta=h$ is not constant, which is absurd.

In conclusion, we can state the following result

\begin{proposition}
  Let $(P,\sigma)$ be a parabolic manifold and $h:P\rightarrow \mathds{R}$ a non-constant positive function. Then, there exists a point $p\in M=P\times_h\mathds{R}$ where $\overline{{\rm Ric}}(\partial_r,\partial_r)<0$.
\end{proposition}

\section{ New Moser-Bernstein theorems}
Let us finally show how previous geometric results are translated in terms of suitable non-linear elliptic partial differential equations. Firstly, from Corollary~\ref{cor:minimal} a Moser-Bernstein result for the solutions to~\eqref{eq:PDE} is derived.

\begin{theorem} \label{thm:MB} Let $(P,\sigma)$ be a Riemannian parabolic $n$-manifold and let $h\in C^\infty(P)$ be a positive smooth function, such that $0<\inf(h)\leq \sup(h)<\infty$. Then the only half-bounded entire solutions $u$ to the equation 

\[\mathrm{div} \left(  \frac{h \nabla u}{ \sqrt{1 + h^2 |\nabla u|_{\sigma}^2}} \right) +\frac{\sigma (\nabla h, \nabla u)}{\sqrt{1 + h^2|\nabla u|_{\sigma}^2}}=0,\]
with $\left|\nabla u\right|_{\sigma}<C$ for some constant $C\in\mathds{R}$, are the constant functions.
\end{theorem}

Moreover, taking into account the proof of Theorem~\ref{thm:2} we can also state a non-existence result.

\begin{theorem} Let $(P,\sigma)$ be a Riemannian parabolic $n$-manifold, $h\in C^\infty(P)$ a positive smooth function such that $0<\inf(h)\leq \sup(h)<\infty$, and $H\geq 0$ (resp. $H\leq 0$) be a non-identically zero smooth function defined on $P$. Then do not exist  entire bounded solutions from above (resp. from below) to the equation 
\[\mathrm{div} \left(  \frac{h \nabla u}{ \sqrt{1 + h^2 |\nabla u|_{\sigma}^2}} \right) +\frac{\sigma (\nabla h, \nabla u)}{\sqrt{1 + h^2|\nabla u|_{\sigma}^2}}=nH,\]
with $\left|\nabla u\right|_{\sigma}<C$ for some constant $C\in\mathds{R}$.
\end{theorem}

\section{Acknowledgements}
The first author is partially supported by MINECO/FEDER project references MTM2015-65430-P and PGC2018-097046-B-100, Spain, and Fundaci\'on S\'eneca project reference 19901/GERM/15, Spain. Her work is a result of the activity developed within the framework of the Program in Support of Excellence Groups of the Regi\'{o}n de Murcia, Spain, by Fundaci\'{o}n S\'{e}neca, Science and Technology Agency of the Regi\'{o}n de Murcia. The second author is partially supported by the Spanish Grant MTM2016-78807-C2-2-P (MINECO and FEDER funds). The third author is partially supported by the Spanish MINECO and ERDF project Grant MTM2016-78807-C2-1-P.

\bibliographystyle{siam}

\end{document}